\documentclass[12pt]{amsart}

\usepackage{fullpage}

\usepackage[table]{xcolor}
\usepackage{newunicodechar}
\usepackage[utf8]{inputenc}
\usepackage[T1]{fontenc}
\usepackage{fontenc}

\usepackage{graphicx} 

\usepackage{amsfonts,lmodern}
\usepackage{amssymb,mathabx}
\usepackage{ifthen}
\usepackage{pstricks}
\usepackage{latexsym,eurosym,ulem}
\usepackage{footmisc}

\usepackage{alltt}
\usepackage{calrsfs}
\usepackage{fancyhdr}
\usepackage{graphicx}

\usepackage{amssymb}
\usepackage{amsmath}
\usepackage{mathabx}
\usepackage{enumerate}
\usepackage[colorlinks=true,linkcolor=black,citecolor=black,urlcolor=black]{hyperref}
\usepackage[all,cmtip]{xy}
\usepackage{alltt}
\usepackage{footmisc}

\font\rurm=wncyr10 scaled \magstep1

\def\sha{{{\textnormal{\rurm{Sh}}}}}

\def\Sha{{\sha}^2}

\def\Z{{\set Z}}

\def\Q{{\set Q}}

\def\Sl{{\rm SL}}

\def\Z{{   \mathbb Z }}
\def\Q{{\mathbb Q}}

\def\F{{\mathbb F}}

\def\E{{\mathcal E}}

\def\p{{\mathfrak q}}
\def\P{{\mathfrak Q}}

\def\q{{\mathfrak q}}
\def\tq{ \tilde{ \mathfrak q}}

\def\rb{\overline{\rho}}

\def\bU{\overline{U}}

\def\bvarphi{\overline{\varphi}}

\def\Gov{{\rm Gov}}

\def\gta{\Gamma^{ta}_K}
\def\ga{\Gamma_K}

\def\sha{{{\textnormal{\rurm{Sh}}}}}
\def\CyB{{{\textnormal{\rurm{B}}}}}

\def\Sha{{\sha}^2}

\newtheorem{Theorem}{Theorem}
\newtheorem*{Definition}{Definition}

\newtheorem{theo}{Theorem}[section]
\newtheorem*{conjectureNC}{Conjecture}

\newtheorem{coro}[theo]{Corollary}
\newtheorem{prop}[theo]{Proposition}
\newtheorem{lemm}[theo]{Lemma}
\newtheorem{exem}[theo]{Example}
\newtheorem{rema}[theo]{Remark}

\begin{document}

\parindent=0cm

\title{On the strong Massey property for number fields}

\author{Christian Maire}
 \address{FEMTO-ST Institute, Universit\'e de Franche-Comt\'e, CNRS, 15B avenue des Montboucons, 25000 Besan\c con, France}
\email{christian.maire@univ-fcomte.fr}

\author{J\'an Min\'a\v c}
\address{Department of Mathematics, Western University, London, Ontario, N6A 5B7, Canada}
\email{minac@uwo.ca}

\author{Ravi Ramakrishna}
\address{Department of Mathematics, Cornell University, Ithaca, NY 14853-4201 USA}
\email{ravi@math.cornell.edu}

\author{Nguy$\tilde{\text{\^{E}}}$N Duy T\^an}
  \address{Faculty Mathematics and 	Informatics, Hanoi University of Science and Technology, 1 Dai Co Viet Road, Hanoi, Vietnam } 
\email{tan.nguyenduy@hust.edu.vn}

\date{\today}


\thanks{The authors thank Alexander Merkurjev and Federico Scavia for related  discussions with Christian Maire and Jan Min\'a\v c during the {\it Workshop on  Galois Cohomology and Massey Products} in June 2024. 
In particular we are grateful to Federico Scavia for bringing our attention to related work in progress  of Peter Koymans. We are also grateful 
to Koymans
for subsequent discussion of these topics.
The first, seccond and  third authors gratefully acknowledge the support of the
the Western Academy for Advanced Research (WAFAR) during the year 2022/23 and for support during a summer 2023 visit. The first author was also partially supported by the EIPHI Graduate School  (contract “ANR-17-EURE-0002") and by the Bourgogne-Franche-Comté Region.  The second author was partially supported by the Natural Sciences Engineering and Research Council of Canada (NSERC), grant R0370A01. The third author was partially supported by Simons Collaboration grant
\#524863. The fourth author was partially supported by the Vietnam National
Foundation for Science and Technology Development (NAFOSTED) under grant number 101.04-2023.21. }
\dedicatory{In memory of Nigel Boston}

\begin{abstract} Let $n\geq 3$. We show that for every number field $K$ 
with $\zeta_{p} \notin K$, 
the absolute and tame Galois groups $\ga$ and $\gta$ of $K$
satisfy 
the strong $n$-fold Massey property relative to~$p$.
 Our work is based on an  adapted version  of the proof of the Theorem of Scholz-Reichardt.  
\end{abstract}

\maketitle

Fix $K$  a number field and an algebraic closure  $\overline{K}$. We set  $K^{ta} \subset \overline{K}$ to be  the maximal tamely ramified Galois extension of $K$, 
that is 
$K^{ta}$ is the composite of all 
number fields $L\subset \overline{K}$ such that the ramification index $e_{\P}$ at all primes ${\P}$ of $L$ is prime to the residue characteristic of $\P$.
Set $\Gamma_K:=Gal(\overline{K}/K)$, and $\Gamma^{ta}_K=Gal(K^{ta}/K)$. 

\smallskip

Let $p$ be a prime number such that $\zeta_p$, a primitive $p$th root of unity, is not in  $K$.
In \cite{HLMR} the authors  use embedding techniques to characterize finitely generated pro-$p$ groups that can be realized as quotients  of $\gta$. They introduced   the notion of  {locally inertially generated}  pro-$p$ groups  for which congruence subgroups of $\Sl_m(\Z_p)$  are archetypes. 
This key notion provides compatibility with local {tame  liftings} as used in the Scholz-Reichardt theorem (see \cite[Chapter 2, \S 2.1]{Serre}). 
This strategy has implications for Massey products as well.

\medskip

Let $n \geq 3$ and
$U_{n+1}$ be the group of all upper-triangular unipotent $(n + 1)\times (n + 1)$-matrices
with entries in $\F_p$. Let $Z_{n+1}=\langle E_{1,n+1} \rangle$ be the subgroup of all such matrices with all off-diagonal
entries  $0$ except at position $(1, n + 1)$; it is the center of $U_{n+1}$. Set $\bU_{n+1}:=U_{n+1}/Z_{n+1}$ to be the quotient. 
To $\Gamma$  a profinite group and
a continuous homomorphism $\rho : \Gamma \rightarrow  U_{n+1}$ 
with $1 \leq i < j \leq n + 1$, we associate the functions
$\rho_{i,j} : \Gamma \rightarrow \F_p$ giving the $(i, j)$-coordinate. We use similar notation
for homomorphisms  $\overline{\rho} : \Gamma \rightarrow \overline{U}_{n+1}$. 
Note that $\rho_{i,i+1}$ (resp.,  $\rb_{i,i+1}$) is a
group homomorphism.
Set $$\varphi  :=  (\rho_{1,2},\cdots, \rho_{n,n+1}) \mbox{ so } \varphi(U_{n+1}) = (\Z/p)^n$$  and 
$$\bvarphi  := (\overline{\rho}_{1,2},\cdots, \overline{\rho}_{n,n+1}) \mbox{ so } \bvarphi(U_{n+1}) =(\Z/p)^n.$$

\smallskip

We have the commutative diagram of groups:

$$\xymatrix{ 
1 \ar[r] &Z_{n+1}  \ar[r] & U_{n+1} \ar[r]\ar@{->>}[rd]_{\varphi} &  \bU_{n+1}\ar@{->>}[d]^{\bvarphi}\ar[r] &1  \\ 
& & & (\Z/p)^n }
$$

Let $\chi_1, \cdots, \chi_n \in H^1 (\Gamma, \Z/p)$, and set $$\theta:=(\chi_1, \cdots, \chi_n) : \Gamma \rightarrow (\Z/p)^n.$$ 

The existence of a homomorphic lift of $\theta$ to $\bU_{n+1}$ is related to the existence of a subset of $H^2(\Gamma, \Z/p)$, denoted $\langle \chi_1,\cdots, \chi_n\rangle$ and called the Massey product. 
We will bypass the original definition of the Massey product and instead use  a 
consequence 
which characterizes the `defined' and `vanishing' conditions 
via group representations. 
below. For more details see \cite{Dwyer} and also \cite{HW}, \cite{MS} and \cite{MT}. Note that the definitions of Massey products in \cite{Dwyer} and \cite{MT} differ from those in \cite{HW} and \cite{MS} by a sign.

\begin{Definition} \label{defi:MasseyDef}
Let  $\chi_1, \cdots, \chi_n \in H^1(\Gamma, \Z/p)$.
 The Massey product $\langle \chi_1, \cdots, \chi_n \rangle \subset H^2(\Gamma, \Z/p)$ 
 \begin{enumerate} \item[$-$]  is defined if  $\theta$ lifts to $\bU_{n+1}$, {\it i.e.} $\theta=\bvarphi \circ \rho'$ for some homomorphism $\rho': \Gamma \rightarrow \bU_{n+1}$;
\item[$-$] vanishes if  $\theta$ lifts to $U_{n+1}$, {\it i.e.} $\theta=\varphi \circ \rho$ for some homomorphism $\rho: \Gamma \rightarrow U_{n+1}$.
\end{enumerate}
\end{Definition}

These definitions depend crucially on the ordering of the characters. Also, we do not have {\it a priori}: $\rho' \equiv \rho $ modulo $Z_{n+1}$.

\begin{Definition} \label{defi:MasseyDef2} The profinite group $\Gamma$ satisfies the  strong $n$-fold Massey property (relative to $p$) if 
for all $\chi_1, \cdots, \chi_n \in H^1 (\Gamma, \Z/p)$ such that  $$\chi_1\cup \chi_2= \chi_2\cup\chi_3 =\cdots = \chi_{n-1}\cup \chi_n=0,$$  the Massey  product $\langle \chi_1,\cdots, \chi_n\rangle$ vanishes.
\end{Definition}

Set $$A_n=\{ (\chi_1,\ldots,\chi_n)\mid \langle \chi_1, \cdots, \chi_n \rangle \  {\rm vanishes} \},  \ \  B_n=\{ (\chi_1,\ldots,\chi_n) \mid \langle \chi_1, \cdots, \chi_n \rangle \ {\rm is \ defined}\},$$ $$C_n=\{ (\chi_1, \ldots, \chi_n)\mid  \chi_1\cup \chi_2= \chi_2\cup\chi_3=\cdots =\chi_{n-1}\cup \chi_n =0 \in H^2(\Gamma,\Z/p)\}\cdot$$

One has $A_n \subset  B_n \subset C_n$. That $A_n \subset B_n$ follows from Definition 3.1. That $B_n \subset C_ n$ follows from a simple argument - see \cite[Remark 2.2]{MS}. For $n= 3$, $B_3 = C_3$. Other inclusions may be strict in general. 

\medskip

For $p=2$
and $\Gamma_K$ the absolute Galois group of a number field $K$,
Hopkins and Wickelgren \cite{HoWi} 
have shown the remarkable result that
the triple Massey product  
 vanishes
whenever it is defined.  
In \cite{MT} this is established  for $\Gamma_F$
the absolute Galois group of any field $F$.
Harpaz and Wittenberg \cite{HW} have recently proved  the
Min\'a$\check{\rm c}$-T\^an Conjecture for number fields $K$:  
$\ga$ 
satisfies the $n$-fold Massey  property for $p$, that is $A_n=B_n$.

\medskip

 If  a primitive $p$th-root of unity is in a number field $K$, there are counterexamples to the strong  $n$-fold Massey property for $\Gamma_K$,
 that is there are examples where  $B_n \subsetneq C_n$ so we do not have $A_n=B_n=C_n$. 
In Wittenberg's appendix to  \cite{GuiMinHar}
   there is an interesting example discovered by Harpaz and Wittenberg. 
   For $K=\Q$ and $p=2$ the $4$-fold Massey $\langle 34,2,17,34\rangle $ is not defined despite the fact that Hilbert symbols $(34,2) ,(2,17), (17,34)$ vanish (by Kummer theory  we have replaced elements of $H^1(\Gamma_\Q, \Z/2)$ by elements of $\Q^\times$). This however cannot happen in the nondegenerate case, that is when 
   the span of the $\chi_i$ is $4$-dimensional.  See Theorem 6.2 and Remark 6.3 of \cite{GuiMinHar}.
This example was generalized 
by Merkurjev-Scavia \cite[\S 5]{MS} in the context of $4$-fold Massey products 
$\langle bc,b,c,bc \rangle$ for $p = 2$. Thus for  $K=\Q$ and $p=2$  
the Massey product $\langle 13\cdot 17,13,17, 13\cdot 17 \rangle$ is not defined:   $\gta$ does not verify the strong fourfold Massey property, i.e. $B_4 \subsetneq  C_4$.

\medskip

The main point of this paper is that
when $\zeta_p \notin K$, the situation is much nicer:

\begin{Theorem} Take $n\geq 3$, and suppose  that $\zeta_p \notin K$. The profinite  groups $\Gamma_K$ and $\gta$ satisfy the  strong $n$-fold Massey property relative to $p$,  that is $A_n=B_n=C_n$.
\end{Theorem}

We obtain this theorem by giving a  global lifting result (Theorem \ref{maintheorem_bis}) in the spirit of the inverse Galois problem over number fields for $p$-groups $U$ with local conditions, as developed in   \cite[IX, \S 5, Theorem 9.5.5]{NSW} or  \cite[Main Theorem]{Neukirch}. 
 When compared to the main theorem of Neukirch in \cite{Neukirch}, our proof is more explicit, constructive and streamlined to our specific Galois representations.  The notion of local plans as used in \cite{HLMR} is central. We use
 Chebotarev's theorem repeatedly, usually applied simultaneously to a governing field and the part of the tower we have already constructed in our inductive process to provide us with the primes of $K$ that we need.
 The hypothesis $\zeta_p \notin K$ is important throughout this paper. It
 implies our field 
 extensions are linearly disjoint  over $K$ so
 we can choose primes of $K$ to split in these fields as we need.
\smallskip

We can strengthen the theorem above 
by showing that $\theta$ lifts, for any $r\geq 1$, to a subgroup of $Gl_{n+1}(\Z/p^{r})$. Let $\pi_r : Gl_{n+1}(\Z/p^r) \rightarrow Gl_{n+1}(\F_p)$ be the mod $p$ reduction homomorphism.

\begin{Theorem}
Take $\Gamma=\gta$ or $\ga$, and suppose $\zeta_p\notin K$. 
 For $n\geq 3$, let $\theta : \Gamma \rightarrow (\Z/p)^n$  satisfy  $C_n$.
 Let
 $\rho$ be  given by Theorem $1$,
 where we choose all 
 tame primes $\q'$ from that proof to satisfy $N(\p') \equiv 1$ modulo $p^{m(r)}$, where $p^{m(r)}$ is the exponent of $U_{n+1}(\Z/p^r)$. This is possible as $\zeta_p \notin K$. Let $\pi_r:Gl_{n+1}(\Z/p^r) \to GL_{n+1}(\F_p)$ be the natural projection.
  \\
 (i) Then for every $r\geq 1$, there exists a homomorphism $\rho_r : \Gamma \rightarrow Gl_{n+1}(\Z/p^r)$  such that 
 $\pi_r \circ \rho_r = \rho$
 and $ \theta=\varphi\circ  \pi_r\circ  \rho_r$.\\
 (ii) If moreover $\zeta_{p^{r}}\in K_\p$ for every ramified prime $\p$ in $\theta$
 then  $\rho_r$ can be taken such that $\rho_r(\Gamma) \subset U_{n+1}(\Z/p^r)$.
\end{Theorem}

\medskip

Given a prime number $p$, set $K'=K(\zeta_p)$. 
Since $-1=\zeta_2 \in K$, we assume that  $p$  is odd. In particular, archimedean places play no role in our work. Almost all cohomology groups $H^i(G,\Z/p\Z)$ have $\Z/p\Z$-coefficients with trivial action so in those cases we  simply write $H^i(G)$.


 \section{Tools for the Embedding problem} \label{section_embedding}

\subsection{Realizing cyclic extensions with given ramification and splitting} \label{section_pelementary}

The problem of realizing the group $G:=(\Z/p)^d$ as a quotient of $\ga$ which satisfies certain ramification conditions can be solved by induction as in   \cite[Chapter 2, \S 2.1]{Serre} or \cite[\S 2.1]{HLMR}.  This involves a governing field $\Gov_{K,T}$ which controls  the obstructions of our embedding problem (see Proposition \ref{proposition_lifting}). 

Given a finite set $T$ of finite primes of $K$, set $$V^T=\{x\in K^\times; v_\p(x) \equiv 0\ {\rm mod} \ p, \ \forall \p \notin T\}.$$
 Denote by $\Gov_{K,T}$ the governing field $\Gov_{K,T}:=K'(\sqrt[p]{V^T}).$ 
 $$\xymatrix{ & & \Gov_{K,T}:=K'(\sqrt[p]{V^T}) \\
 & K':=K(\zeta_p) \ar@{-}[ur] \\
 & K \ar@{-}[u] }$$
By Kummer theory we see $\Gov_{K,T}/K'$ is an elementary abelian $p$-extension. One easily computes it is unramified outside $T \cup \{ {\mathfrak p} |p \}$. 
Moreover  $\Gov_{K,T}/K$ is a Galois extension with Galois group isomorphic to the semi-direct product $Gal(K'(\sqrt[p]{V^T})/K') \rtimes Gal(K'/K)$, where the action  on $Gal(K'/K)$ is given by Kummer duality: since $Gal(K'/K)$ acts trivially on $V^T$,
it acts via the  cyclotomic character (which is nontrivial as $\zeta_p \notin K$) on the Galois group over $K'$ of each cyclic degree $p$ extension $M/K'$ in $K'(\sqrt[p]{V^T})/K'$. See \cite[Chapter I, Theorem 6.2]{gras}.

\medskip

For a tame prime $\p \notin T$ (that is $\p\nmid (p)$), we write $\sigma_\p$ for the Frobenius  in $Gal(\Gov_{K,T}/K')$ for a fixed prime $\P$ above $\p$. One has (see 
 \cite[Chapter V, Corollary 2.4.2]{gras}):
 
\begin{theo}[Gras]
\label{theo:Gras-Munnier}  
Let $S=\{\p_1,\cdots, \p_s\}$ be a set of primes of $K$ coprime to $T$  and $p$. 
 There exists a cyclic degree $p$ extension $L/K$ exactly ramified at $S$ and splitting completely at $T$  if and only if there exist $a_i\in \F_p^\times$, $i=1,\cdots, s$, such that 
$$\sum_{i=1}^s a_i \sigma_{\p_i}=0 \in Gal(\Gov_{K,T}/K').$$
\end{theo}

Hence, if a tame $\p$ splits completely in $\Gov_{K,T}/K'$, there exists~an~$\Z/p$-extension  $L/K$ exactly ramified at $\p$ and splitting completely at~$T$. 

\begin{rema}\label{rema_local_embedding}
The Frobenius is actually associated to a 
prime $\P$ of $\Gov_{K,T}$ above $\p$, but changing $\P$ changes the Frobenius
by a  power that is not a multiple of $p$. This follows from  our description of $Gal(K'(\sqrt[p]{V^T})/K)$ above and does not affect the condition of Theorem \ref{theo:Gras-Munnier}. Hence we abuse notation and write $\sigma_\p$.
Later we will also use the governing
field $K'(\sqrt[p]{V_N})$ where
$V_N=\{x\in K^\times; v_\p(x) \equiv 0\ {\rm mod} \ p;\, \q \in N \implies x \in (K_\p^\times)^p \}.$
\end{rema}


 \subsection{Cohomology and embedding problems} \label{Section_embedding}

 Let $G$ be a $p$-group of $p$-rank $d$.  Suppose given $H\simeq \Z/p$ a normal subgroup of $G$ such that $G/H$ also has $p$-rank $d$.
 Let $\Gamma$ be a pro-$p$ group, and 
 let   $\overline{\rho} : \Gamma \twoheadrightarrow G/H$ be a surjective morphism. 

 \smallskip
 
We consider  the embedding problem $(\E)$:
  
  $$\xymatrix{ & & & \Gamma \ar@{.>}[ld]_-{?\rho} \ar@{->>}[d]^{\overline{\rho}} \\
1 \ar[r] &H \ar[r] & G \ar@{->>}[r]_-{\pi} &  G/H}$$
where  $\pi$ is the natural projection.
That $G$ and $G/H$ have the same $p$-rank  implies  $H\subset  G^p[G,G]$,  so every solution of $(\E)$ is {\it proper}, that is if $\rho$ exists,  it is surjective. 

\smallskip

Let $\varepsilon$ be the element in $H^2(G/H,H) =H^2(G/H,\Z/p)=H^2(G/H)$ corresponding to the group extension:
\begin{eqnarray}\label{es1} 1\longrightarrow H \longrightarrow G \longrightarrow G/H \longrightarrow 1.\end{eqnarray}
As $d=d(G)=d(G/H)$, we have $\varepsilon \neq 0$.
Let $Inf$ be the the inflation map from $H^2(G/H)$ to $H^2(\Gamma)$. 
The action of $\Gamma$ on $H=\Z/p$ is induced by $\overline{\rho}$ and is thus trivial.

\begin{theo} \label{theorem_embeddin}  The embedding problem $(\E)$ has a  solution if and only if $Inf(\varepsilon)=0$. 
Moreover,  any solution  is always proper.
The set of solutions (modulo equivalence) of $(\E)$ is a principal homogeneous space under $H^1(\Gamma)$.
\end{theo}

\begin{proof}
Proposition 3.5.9 and 3.5.11 of \cite{NSW}.
\end{proof}

\begin{rema}   
 For a prime $\p$ of $K$, denote by $\Gamma_\p$ the absolute pro-$p$ Galois group of $K_\p$. 
We  need to study the local embedding problems attached to  local maps $\iota_\p : \Gamma_\p \rightarrow \Gamma$. Let $\overline{D}_\p \subset G/H$ be the image of 
$\overline{\rho_\p}:=\overline{\rho} \circ \iota_\p$, and $M_\p \subset G$ be the inverse image $\pi^{-1}(\overline{D}_\p)$. We have the local embedding problem $(\E_\p)$:
  
  $$\xymatrix{ & & & \Gamma_\p \ar@{.>}[ld]_-{?\rho_\p} \ar@{->>}[d]^{\overline{\rho_\p}} \\
1 \ar[r] &H_\p  \ar[r] & M_\p \ar@{->>}[r]_-{\pi_\p} &  \overline{D}_\p}$$

The set of solutions (modulo equivalence and if it exists) of $(\E_\p)$ is a principal homogeneous space under $H^1(\Gamma_\p)$.
\end{rema}

\subsection{Trivializing the Shafarevich group} \label{section_Q}

 Let $X$ be a finite set of places of $K$. 
Let $K_{X}$ the maximal pro-$p$ extension of $K$ unramified outside $X$ and set $\Gamma_X:=Gal(K_{X}/K)$. 
 
 \medskip
 
  Let  $\Sha_{X}$ be the  kernel of  the localization map $$ H^2(\Gamma_{X}) \longrightarrow \prod_{\p \in X} H^2(\Gamma_{\p}).$$
  Set $V_{X}:=\{x\in K^\times; v_\p(x)  \equiv 0\ {\rm mod} \ p, \ \forall \p; x\in (K_\p^\times)^p , \ \forall \p \in X\}$.
  By the work of  Koch and Shafarevich (see \cite[Chapter 11, Theorem 11.3]{Koch}), we have that $$\Sha_{X} \hookrightarrow \CyB_{X}:= (V_{X}/(K^\times)^p)^\wedge,$$ where 
 the superscript $^\wedge$ indicates the  Pontryagin dual.

  \begin{lemm}\label{lemm_kill_sha}
One can choose $N$ such that  $\CyB_N$ (and therefore $\Sha_{N}$) is trivial.
  \end{lemm}
  
  \begin{proof}
   From the definition of $V_X$ we have that $K'(\sqrt[p]{V_{X}})/K'$ is unramified outside $\{p\}$ and completely split at $X$. Since the maximal elementary $p$-extension of $K'$ unramified outside $\{\mathfrak p|p  \} $ is, by Hermite-Minkowski, finitely generated, we see  for a finite  set $N$ whose Frobenius elements span
   $Gal(K'(\sqrt[p]{V_\emptyset})/K')$ (which exists by Chebotarev's theorem)  that 
    $K'(\sqrt[p]{V_{N}})=K'$. Thus $\forall x\in V_{N}$ we have $x\in (K')^p$.
   By taking the norm of $x$ in $K'/K$ and using the fact that $([K':K],p)=1$, we conclude that $x\in K^p$. Thus $V_{N}/(K^\times)^p=1$ which implies $\Sha_{N}=1$.
\end{proof}
    It is an exercise to see that for any sets $Y,Z$ that
    $V_{Y\cup Z} \subset V_Y$ so $\CyB_Y \twoheadrightarrow \CyB_{Y\cup Z}$. Thus  $V_{N\cup Y}/(K^\times)^p$ and $\Sha_{N \cup Y}$ are trivial for any set $Y$.
\medskip

Henceforth we assume that  $\Sha_{N}=1$ and $\overline{\rho}:\Gamma_N \to G/H$ is given.  Thus if at every $\p \in N$
  there is no local obstruction to lift $\overline{\rho_\q}$ to $\rho_\q: \Gamma_q \to G$, then  the embedding problem $(\E)$ has a solution in $K_{N}/K$.  We have reduced solving the obstruction problem to purely local problems. It is interesting to note that when we work with local plans at $\q \in N$ (see \S\ref{subsection_local}) we can choose them to be unramified at these $\q$. Thus the primes of $N$ force the obstruction problem to be local, but they need not be ramified in our resolution of the Massey problem!
 
 \medskip
 
 The question is: How do we create a situation for which there are no local obstructions for every quotient of $G$? We address this  in  the two next subsections.


\subsection{The local-global principle} \label{section_local_global}

Let $X$ be a finite set of primes of $K$. Given another finite set $R$ of primes of $K$, denote by $\psi_R$ the localization map:
$$ \psi_R : H^1(\Gamma_{X\cup R}) \longrightarrow \prod_{\p \in X} H^1(\Gamma_\p).$$

We will control the image of $\psi_R$ in the case where  
$R=\{ \tilde{\q} \}$,  $\tilde{\q}$ being  a {\it tame} prime. Set $N(\tilde{\q})$ to be the absolute norm of  $\tilde{\q}$.

\medskip

The condition $\zeta_p \notin K$ is needed at this point, in particular the following lemma is crucial to for Proposition~\ref{proposition_lifting}.

\begin{lemm}\label{disjoint}
Let $F/K$ be a $p$-extension. If $\zeta_p \notin K$, then $F(\zeta_p) \cap 
K'(\sqrt[p]{V^X})
=K'$. 
\end{lemm}
\begin{proof} 
   The intersection clearly contains $K'$. If it was larger, there would exist a degree $p$ extension $M/K'$, Galois over $K$ with $M \subset F(\zeta_p)$. Then $Gal(K'/K)$ would act on $Gal(M/K')$ in two different ways: trivially by viewing $M$ in $F(\zeta_p)/K'$, and via the cyclotomic character by viewing $M$ in $K'(\sqrt[p]{V^X})/K'$. These actions are incompatible when $\zeta_p \notin K$. 
\end{proof}

 Recall Proposition 1.4 of \cite{HLMR}.

 \begin{prop} \label{proposition_lifting} 
  Let  $X$ be a finite set of  primes, and let $\displaystyle{(f_\q)_{\q \in X} \in \prod_{\p \in X} H^1(\Gamma_\p)}$. There exist infinitely many finite primes $\tq$  such that $(f_\q)_{\q \in X} \in  Im(\psi_{\{\tilde{\q}\}})$.
  Moreover, when $\zeta_p \notin K$, the primes $\tilde{\q}$ can be chosen such that:
  \begin{enumerate}
   \item[$(i)$]  $\tilde{\q}$ splits completely in $F/K$,   where $F/K$ is a given  $p$-extension, \color{black} 
\item[$(ii)$] the $p$-adic valuation of $N(\tilde{\q})-1$ is given in advance. 
\end{enumerate}
 \end{prop}

 \begin{proof} 
 See \S 1.2 of \cite{HLMR}. 
 \end{proof}

 \begin{rema} \label{rema_1.7} Take $m\geq 1$.
 In the proof of Proposition \ref{proposition_lifting} the tame prime $\tilde{\q}$ is characterized by its Frobenius in $K(\zeta_{p^m},\sqrt[p]{V^X})/K'$ ; in this case we can choose $v_p(N(\tilde{\q})-1)= m$. Thus one can give an upper bound for the absolute norm  of the smallest such $\tilde{\q}$.  Let $d_{X,m}$ be the absolute value of the absolue discriminant of the number field $K(\zeta_{p^{m+1}},\sqrt{V^X})$. 
Then, assuming the GRH, $N(\tilde{\q}) \ll (\log(d_{X,m}))^2 $. See  \cite{LagOdl}.
 \end{rema}


 \subsection{Local plans}\label{subsection_local}
 Previously, we had considered the problem $(\E)$ where $H\simeq \Z/p$. To  prove our main theorem we need to lift 

 $$\xymatrix{ & & & \Gamma \ar@{.>}[ld]_-{?\rho} \ar@{->>}[d]^{\overline{\rho}} \\
1 \ar[r] &V  \ar[r] & U \ar@{->>}[r]_-{\pi} &  U/V}$$
where $V$ is  some normal subgroup of the $p$-group $U$. \color{black}
Of course we will do this one step at a time where each kernel is $\Z/p$, but at each step we will need more ramified primes. As we introduce a new ramified prime, we need a {\it local plan} for it, that is a local solution to the {\it overall} lifting problem above.

 As before  set $N$ is taken so that $\Sha_N=1$.
We suppose given a sub-extension $F/K$ of $K_N/K$ with Galois group isomorphic to $ U/V$, that is we have a homomorphism
$\overline{\rho}:\Gamma_N \twoheadrightarrow U/V$.
 
 \medskip

 Given   $\p \in N$, let $\overline{{\rho}}_{\p} : \Gamma_\p \longrightarrow \overline{D}_{\p} \subset  U/V$ be the restriction of
 $\overline{\rho}$,  where $\overline{D}_{\p}$ is the decomposition group of $\p$ in  $U/V=Gal(F/K)$ (after fixing a prime $\P | \p$).

 \smallskip
 
 We seek a  lift $\rho_\p$ of $\overline{{\rho}}_{\p}$ in $U$, in the  setting where $\overline{{\rho}}_{\p}$ is ramified:

  $$\xymatrix{ &  \Gamma_\p \ar@{.>}[ld]_{?\rho_\p} \ar@{->}[d]^{\overline{{\rho}}_{\p}} \\
   U 
   \ar@{->>}[r]
   & U/V  }$$

   If $\rho_\p$ exists, we say that $\rho_\p$ is a {\it local plan} for $\Gamma_\p$ into $U$.

 \medskip
 
 Recall that the pro-$p$ group $\Gamma_\p$ is
 \begin{enumerate}
  \item[$-$] free when $\zeta_p \notin K_\p$,
  \item[$-$]  Demushkin when $\zeta_p \in K_\p$,
 \end{enumerate}
 
 Let us be more precise.
 
 \smallskip
 
Consider $\q \nmid p$. We suppose that $\zeta_p\in K_\p$ (if not, $\Gamma_\q\simeq \Z_p$).
Recall that in this case  $\Gamma_\p \simeq \Z_p \rtimes \Z_p$ is Demushkin. Indeed, let $\tau_\p \in \Gamma_\p$ be a generator of  inertia and $\sigma_\p$  a lift of the Frobenius. One has the unique relation: $\sigma_\p \tau_\p \sigma_\p^{-1} = \tau_\p^{N(\p)}$. See \cite[Chapter 10, \S 10.2 and \S 10.3]{Koch}.

\smallskip

We now consider $\mathfrak p|p$ and set $n_{\mathfrak p}=[K_{\mathfrak p}:\Q_p]$. If $\zeta_p\notin K_{\mathfrak p}$, then $\Gamma_{\mathfrak p}$ is free pro-$p$ on $n_{\mathfrak p}+1$ generators.
If $\zeta_p\in K_{\mathfrak p}$, then $\Gamma_{\mathfrak p}$ is a Demushkin  on $n_{\mathfrak p}+2$ generators $x_1,\cdots, x_{n_{\mathfrak p}+2}$; in this case the unique relation is $x_1^{p^s}[x_1,x_2]\cdots [x_{n_{\mathfrak p}+1},x_{n_{\mathfrak p}+2}]$, where $p^s$ is the largest power of $p$ such that $K_{\mathfrak p}$ contains the  $p^s$-root of the unity.

 \medskip
We give examples of local plans.

     \begin{exem}\label{local_plan_SR}[S-R plan]  Recall the principle  of the proof of the   Scholz-Reichardt theorem.
      Suppose that $U$ contains an element  $y$ of order  $p^m$. 
     Take a prime $\p$ such that $v_p(N(\p)-1) = m$ - this is possible as $\zeta_p \notin K$.
Suppose we are given a homomorphism $\overline{\rho}_{\p}:\Gamma_\p \longrightarrow U/V$ defined  by $\overline{\rho}_{\p}(\sigma_\p)=\overline{1}$ and $\overline{\rho}_{\p}(\tau_\p)=\bar{y}$.
   Since  $y^{N(\p) -1}=1$, the map   $\rho_\p : \Gamma_\p \rightarrow U$   given by $\rho_\p(\sigma_\p)={1}$ and $\rho_\p(\tau_\p)=y$ is a   homomorphic lift   of $\overline{\rho}_{\p}$ from $U/V$ to $U$.
     \end{exem}

     \begin{exem}\label{triv}[Trivial plan]
There are two trivial plans.

1) Suppose $F/K$ unramified at $\p$, {\it i.e.} let  $\overline{\rho}_{\p}:\Gamma_\p \longrightarrow U/V$ be a homomorphism  defined  by $\overline{\rho}_{\p}(\sigma_\p)=\overline{x}$ for some $\overline{x} \in U/V$ and $\overline{\rho}_{\p}(\tau_\p)=\overline{1}$.
   Let $x \in U$ be any lift of $\overline{x}$.
   The map   $\rho_\p : \Gamma_\p \rightarrow U$   given by $\rho_\p(\sigma_\p)={x}$ and $\rho_\p(\tau_\p)=1$ is a   homomorphic lift of $\overline{\rho}_{\p}$ from $U/V$ to $U$.
   
   2) The previous unramified setting is a special case of the  situation where $\Gamma_\p$ is pro-$p$ free, e.g. if
   $\q \mid p$ and $\zeta_p \notin K_\p$. Any lift works in this case as well.

     \end{exem}

     \begin{exem}\label{example_abelian}[Abelian plan] 
       Suppose that $U$ contains two elements $x$ and $y$ satisfying $xy=yx$. 
       Let $p^\ell$ be the order of $y$.
     Take a prime $\p$ such that $N(\p) \equiv 1 \pmod {p^k}$ with $k\geq \ell$.
     The pro-$p$ part of the abelianization of $\Gamma_\q$ is $\Z_p \times \Z/p^k$.
Suppose given $\overline{\rho}_{\p}:\Gamma_\p \longrightarrow U/V$ defined  by $\overline{\rho}_{\p}(\sigma_\p)=\overline{x}$ and $\overline{\rho}_{\p}(\tau_\p)=\bar{y}$.
   The map   $\rho_\p : \Gamma_\p \rightarrow U$   given by $\rho_\p(\sigma_\p)={x}$ and $\rho_\p(\tau_\p)=y$ is a   homomorphic lift of $\overline{\rho}_{\p}$ from $U/V$ to $U$.
     \end{exem}
     
     There is another important local plan in the context of Massey products  coming from results of Min\'a$\check{\rm c}$-T\^an (\cite[Proposition 4.1]{MT_JNT} and \cite[Theorem 4.3]{MT}).  We call these {\it Massey local plans} and use them  in the proof of Theorem~\ref{Theorem_Massey}. 
     

 \section{A global lifting result}

     The main result of this section is
     Theorem \ref{maintheorem_bis},
     a variation of the Scholz-Reichardt theorem. We start with a proposition
     useful  in proving this theorem in the split case, that is when  $d(U) > d(U/V)$.
     In the context of  the strong Massey property, it is useful for the {\it degenerate case}.

     \begin{prop}\label{proposition_localpropertes} Suppose that $\zeta_p\notin K$. Let $F/K$ be a $p$-extension  and let $S$ be a finite set of primes of~$K$. Let $k,m\geq 1$ and for $i=1,\cdots, k$     let $(\chi_{\p,i})_{\p\in S} \in H^1(\Gamma_\p)$. 
 Then for $i=1,\cdots, k$
\begin{enumerate}
 \item[$(i)$] there exist a  
  $\chi_i \in H^1(\Gamma_K)$  such that for every $\p \in S$, $ {\chi_i}_{|{\Gamma_\p}} =  \chi_{\p,i} $. Let $M_i/K$ be the $\Z/p$-extension fixed by $Ker(\chi_i)$;
 \item[$(ii)$] the extension $M_i/K$ is unramified outside $S\cup \{\q_{i}'\}$, where $\q_{i}'$ is a new tame prime such that $v_p(N(\q_{i}')-1) =m$; 
 \item[$(iii)$]  the extension $M_i/K$ is totally ramified at  $\q_{i}'$, 
 \item[$(iv)$]  for every $i$, $\q_{i}'$ splits completely in  $F/K$.
 \item[$(v)$]  for every $j\neq i$, $\q_{j}'$ splits completely in  $M_i/K$. \color{black}
\end{enumerate}

     \end{prop}

     \begin{proof} 
     Given Proposition~\ref{proposition_lifting}, 
     $(i)$, $(ii)$ and $(iv)$ are perhaps not surprising and $(iv)$ involves a straightforward trick. Establishing $(v)$ is crucial for our results and this requires Gras' result, Theorem~\ref{theo:Gras-Munnier}.

\smallskip

$\bullet$ 
By Proposition \ref{proposition_lifting}, there exists a tame prime $\q'$ such that there exists a $\chi_1 \in H^1(\Gamma_{S\cup \{\q'\}})$ that  $\chi_1 \mid_{\Gamma_\p}=\chi_{\p,1}$ for each  $\p\in S$. This is $(i)$.
Moreover, since $\zeta_p\notin K$, 
using that  $\Gov_{K,S}/K'$ and $F(\zeta_{p^{m+1}})/K'$ \color{black} are linearly disjoint,
we can choose  $\q'$  such that $v_p(N(\q')-1)=m$  and    $\q'$ splits completely in $F/K$. 
Set $M$ to be the $\Z/p$-extension of $K$ fixed by~$\chi_1$. 

If $\chi_1$ is ramified at $\q'$, then we set $M_1:=M$ and $\q_1':=\q'$. 

If $\chi_1$ is unramified at $\q'$,  we choose a tame prime $\q_{1}'$ that  splits completely in   $\Gov_{K,S} F (\zeta_{p^{m}})/K$ but does not split completely in $K(\zeta_{p^{m+1}})/K$.  \color{black}By Theorem \ref{theo:Gras-Munnier} there exists a $\Z/p$-extension $M_1'/K$ exactly ramified at $\q_{1}'$ in which the places of $S$ split  completely. Then 
$Gal(M'_1M/K)\simeq \Z/p \times \Z/p$, we choose $M_1$ to be any intermediate extension other than than $M$ and $M'_1$. The field
$M'_1$ satisfies $(ii)$, $(iii)$, and $(iv)$, but we must exclude it get $(i)$.
 We have established $(i)-(iv)$.

\medskip

$\bullet$ Set $S_1=S\cup \{\q_1'\}$. 
Set $\chi_{\q_1'} \in H^1(\Gamma_{\q'_{1}})$ to be
trivial.  By Proposition \ref{proposition_lifting}, there is a tame prime $\q'$ such that there exists a $\chi \in H^1(\Gamma_{S_1\cup \{\q'\}})$ 
with $\chi |_{\Gamma_\q} =
\chi_{\p,2}$ for every $\p\in S_1$. 
This is $(i)$. 
Moreover, since $\zeta_p\notin K$, the prime $\q'$ can be chosen such that $v_p(N(\q')-1)=m$, and such that $\q'$ splits completely in $FM_1/K$ (indeed $\Gov_{K,S}/K'$ and $FM_1(\zeta_{p^{m+1}})/K'$ are linearly disjoint).
As before,  set $M$ to be the $\Z/p$-extension of $K$ corresponding to $\chi$. By the choice of $\chi_{\q_1'}$, observe that $\q'_{1}$ splits completely in $M/K$. \color{black}

If $\chi$ is ramified at $\q'$, set $M_2=M$ and $\q_2':=\q'$. 

If $\chi$ is unramified 
we proceed as did above to get $\q_{2}'$ and $M_2$.

Then, in all case, one has $(i)-(iv)$.

\smallskip

Note that the splitting choices on  $\q_1'$ and $\q_2'$ give $(v)$ as well.

\medskip

Repeat this process with $S_2=S_1 \cup \{ \q'_2\}$ to find $\q_{3}'$ and $M_3$ etc. to finish the proof.
\end{proof}

Theorem~\ref{maintheorem_bis} is the key result we need to establish the strong $n$-fold Massey property for $\Gamma_K$ and $\Gamma^{ta}_K$. The proof of Theorem~\ref{maintheorem_bis}
is more involved than the corresponding argument of \cite{HLMR} or the proof of Scholz-Reichardt for $U_{n+1}$. This is because in those situations one starts with a trivial homomorphism and inductively builds the entire group. The local plans are easy to introduce and maintain. In 
\S~\ref{subsection:proof}
we {\it start} with a homomorphism 
$\theta:\Gamma \to (\Z/p)^n $ and must lift that to $\bU_{n+1}$ and $U_{n+1}$, rather than build it ourselves.

     \begin{theo}\label{maintheorem_bis}
     Suppose that $\zeta_p\notin K$.
      Let $U$ be a $p$-group,  and $V\lhd U$ be a normal subgroup of $U$.
      Let $F/K$ satisfy  
      \begin{itemize}
          \item $F/K$ is unramified outside a set of primes $S=\{ \p_1,\cdots, \p_t\}$;
          \item for each $\p_i \in S$ the  decomposition group $\overline{D}_{\p_i}$ in $F/K$ respects a local plan $\rho_{\p_i}: \Gamma_{\p_i} \rightarrow U$. In other words, ${D}_{\p_i} \equiv \rho_{\p_i}(\Gamma_{\p_i})$ modulo $V$.
          \item $Gal(F/K)\simeq U/V$.
          
      \end{itemize} 
      
      Then there exists a Galois extension $L/K$ in $\overline{K}/K$ such that:
      \begin{enumerate} 
      \item[$(i)$] $L/K$ contains $F/K$;
       \item[$(ii)$] $Gal(L/K)\simeq U$;
       \item[$(iii)$] $\rho_{\p}(\Gamma_{\p})  \simeq  D_{\p}:= Gal(L_{\p}/K_{\p})$, for every $\p \in S$.
      \end{enumerate} 
       Moreover, if $F \subset K^{ta}$ then $L$ can be chosen in  $K^{ta}$.
     \end{theo}

\begin{proof} Let $p^m$ be the exponent of $U$.
Since $\zeta_p \notin K$, Lemma~\ref{disjoint} implies  $F(\zeta_p) \cap \Gov_{K,S}=K'$. \color{black} This will allow us to use Chebotarev's theorem to choose primes that split as we need in $K(\zeta_{p^{m+1}})$, $F$ and $\Gov_{K,S}$.
\smallskip

We first solve the problem when the group extension $1\to V \to U \to U/V \to 1$ is split.

\smallskip

$\bullet$ 
Let $d$ be the $p$-rank of $U$, and let $d_0$ be the $p$-rank of $U/V$. Set $k=d-d_0$.

Since we are in the split setting, $k>0$. 
 Hence $V \nsubset U^p[U,U]$: there exists a maximal subgroup $U_1$ of $U$ such that $V \nsubset U_1$. By maximality $U_1 \lhd U$ and we have
  $$U/(V\cap U_1) \hookrightarrow U/V \times U/U_1.$$ On the other hand  $$ 1 \longrightarrow  V/(V\cap U_1)  \longrightarrow U/(V\cap U_1) \longrightarrow U/V \longrightarrow 1$$ and $V/(V\cap U_1) \simeq VU_1/U_1 \hookrightarrow U/U_1 \simeq \Z/p$. Since $V \nsubset U_1$, we see 
$$|U/(V\cap U_1)|= |U/V| |U/U_1|,$$ and then $$g_1: U/(V\cap U_1) \stackrel{\simeq}{\longrightarrow} U/V \times U/U_1 \simeq U/V \times \Z/p \cdot $$

Set $A_1:=V\cap U_1$ so  $U/A_1$ has $p$-rank $d_0+1$. Set $k'=d-(d_0+1)$. If $k'>0$ then, as before,  there exists a maximal subgroup $U_2$ of $U$ such that $A_1 \nsubset U_2$, and then
$$g_2: U/(A_1\cap U_2) \stackrel{\simeq}{\longrightarrow} U/A_1 \times U/U_2 \simeq U/A_1 \times \Z/p \cdot $$

 We continue the process 
and set $A=V \cap U_1 \cap \cdots \cap U_k$ so
$$ \label{se2} g:  U/A \stackrel{\simeq}{\longrightarrow} U/V \times U/U_1 \times \cdots \times U/U_k \simeq U/V \times (\Z/p)^k\cdot$$ 

\medskip

For $i=1,\cdots, k$, let $x_i \in U$ such that  $U/U_i=\langle x_i U_i\rangle \simeq \Z/p$. 

\medskip

For $i=1,\cdots, k$, let $\eta_i \in H^1(U/A)$ be 
defined by $\eta_i(x_j)=\delta_{i,j}$, and $\eta_i$ is trivial on $U/V$.  The restriction
${\eta_i}_{|\rho_\p(\Gamma_\p)}$ can be viewed as an element  of $H^1(\Gamma_\q)$ and thus an input
of Proposition~\ref{proposition_localpropertes}.

By Proposition \ref{proposition_localpropertes}, 
there exist $\chi_i \in H^1(\Gamma_K)$ for $i=1,\cdots, k$ such that
\begin{enumerate}
 \item[$(i)$] for every  $\p \in S$, $ {\eta_i}_{|\rho_\p(\Gamma_p)} =  {\chi_i}_{|\Gamma_\p} $;
 \item[$(ii)$] for each $i$
 let $M_i$ the $\Z/p$-extension fixed by $Ker(\chi_i)$.
 The extension $M_i/K$ is unramified outside $S\cup \{\q_{i}'\}$ where $\q_{i}'$ is a new tame prime, such that $v_p(N(\q_{i}')-1)=m$.
 \item[$(iii)$]  the extension $M_i/K$ is totally ramified at  $\q_{i}'$,
 \item[$(iv)$]  for every $i$, $\q_{i}'$ splits completely in  $F/K$.
 \item[$(v)$]  for every $j\neq i$, $\q_{j}'$ splits completely in  $M_i/K$. 
\end{enumerate}

Put $K_2:=FM_1\cdots M_k$. 
By $(iii)$ and $(iv)-(v)$, one gets 
$$ h:  Gal(K_2/K) \stackrel{\simeq}{\longrightarrow}  Gal(F/K) \times (\Z/p)^k  \stackrel{\simeq}{\longrightarrow}  U/V \times (\Z/p)^k    \stackrel{\simeq}{{\longleftarrow}} U/A : g $$

\medskip

 Condition $(i)$ above implies
the two isomorphisms $g$ and $h$ respect the initial  local plans $\rho_\p$ for every $\p \in S$:  the image of
$\rho_\p(\Gamma_\p) \subset U$ projects to  to $D_\p(K_2/K)$ in $U/A \simeq Gal(K_2/K)$.
Moreover, for the other ramified primes $\q_i'$, one has  $v_p(N(\q_i') -1)=m$, and $D_{\q_i'}(K_2/K)=I_{\q_i'}(K_2/K) \simeq \Z/p$, $i=1,\cdots, k$. 

\smallskip

The extension $K_2/K$ is unramified outside  $S_2:=S\cup\{\q_1',\cdots, \q_k'\}$. Moreover, we have a local plan for each of these primes:  the one given by the hypothesis for those in~$S$, and  the S-R local plan for the $\q'_i$ (see Example \ref{local_plan_SR}).

\medskip

As $A$ is  contained in  the Frattini subgroup of $U$, the proof of the split case is  done.  Note that
all the $\q_{i}'$ are 
\medskip

We now proceed by induction 
when $1\to V \to U \to U/V \to 1$ does not split.

\smallskip
 
$\bullet$  Let $H_2 \subset U^p[U,U]$ be normal in $U$ and consider a sequence $H_n \subset H_2$, $n\geq  2$, of normal subgroups of~$U$, such that $H_n/H_{n+1} \simeq \Z/p$ and  $H_n=U$ for $n\gg 0$:  this is always possible since $U$ is a $p$-group.

Set $\Gamma=\ga$ or $\gta$.
 Consider  the embedding problem $(\E_n)$:
  
  $$\xymatrix{ & & & \Gamma \ar@{.>}[ld]_-{?\rho_{n+1}} \ar@{->>}[d]^{{\rho_n}} \\
1 \ar[r] &H_n/H_{n+1}  \ar[r] & U/H_{n+1} \ar@{->>}[r]_{g_n} &  U/H_n}$$
where  ${g_n}$ is the natural projection.

Let $N_2$ be a finite set of primes containing those ramified in $K_2/K$ ({\it i.e.} $S_2\subset N_2$) and such that  $\Sha_{N_2}=1$.

\medskip

At each level $U/H_n$ we want to:
\begin{enumerate}
 \item[$(i)$]  solve the {\it nonsplit} embedding problem $(\E_n)$;
 \item[$(ii)$]  adjust the solution by an element of $H^1$  (adding  ramification at a new prime) such that the new solution is on all local plans, including at the new prime of ramification. There is then no local obstruction for the next step of the induction.
\end{enumerate}

\smallskip

$-$ By construction 
and that $\sha_{N_2}=0$, there is no obstruction to lift the decomposition group $D_{\p}$ of $\p$ in $U/H_2$ to $U$: for each $\p$, one has a local plan. (If $\p \in N_2 \backslash S_2$ take the trivial plan (1) of
Example~\ref{triv}.)

One can apply \S \ref{section_Q}:
there exists a $\Z/p$-extension of $K_3/K_2$, unramified outside $N_2$, Galois over $K$, solving the lifting problem ($\E_2$) to $U/H_3$. That is $(i)$ of the strategy.

\smallskip

$-$
The problem now is that the decomposition group $D_\p$ at $\p \in N_2$ in $K_3/K$ may be off the local plan and therefore it may not be liftable to $U/H_4$. If we are on all local plans in $N_2$, we proceed as we did previously to lift to $U/H_4$.

Assume now that we are not on all local plans.
As $H^1(\Gamma_\p)$ acts as a principal homogeneous spaces on the solutions to our local embedding problem $(\E_\p)$, the existence of a local plan implies the existence of $f_\q \in H^1(\Gamma_\p)$ by which we can adjust our solution to be on the local plan.

\smallskip

The quotient $H_2/H_3$ is generated by the image of an element $y\in U$ of order $p^{m_0}$ with $m_0\leq m$.

We now use Proposition \ref{proposition_lifting}
to find a  tame place~$\q_2$  such that 
for $R_2=\{\q_2\}$
$$(f_\q)_{\q \in N_2}\in Im(\psi_{R_2}),\,\,\,   v_p(N(\q_2)-1)=m,$$ and $\q_2$ splits completely in $K_2/K$.
Hence there exists an element of 
$H^1(\Gamma_{N_2\cup R_2})$
that puts us on the local plan for all 
$\p \in N_2$.
As $\q_2$ splits completely in $K_2/K$, we are on the S-R local plan for
$\q_2$. Set $N_3= N_2 \cup \{\q_2\}$.
We are on the local plan at all $\q \in N_3$ and can proceed by induction.

As in the split case, 
all new primes of ramification are tame,
 so if $F\subset K^{ta}$, then $K_2 \subset K^{ta}$.
\end{proof}



\section{Applications} \label{section_massey}

\subsection{The main result} \label{subsection:proof}

In this section we  prove:

\begin{theo} \label{Theorem_Massey} 
Let $K$ be a number field and let $p$ be a prime number such that $\zeta_p \notin K$. Let $n\geq 3$. Then the profinite groups $\gta$ and $\ga$ satisfiy the strong $n$-fold Massey property (relative to~$p$).
\end{theo}

Let $\chi_1, \cdots, \chi_n \in H^1 (\Gamma)$, and set $$\theta:=(\chi_1, \cdots, \chi_n) : \Gamma \rightarrow (\Z/p)^n.$$ 
 
Let $F:= K(\theta) \subset \overline{K}$ 
be the fixed field the kernel of $\theta$, 
that is intersection of the kernels  $Ker(\chi_i)$. Note $K(\theta) \subset K^{ta}$.
Set $G_0=Gal(F/K)$. Then  $G_0 $ is isomorphic to $(\Z/p)^r$. 

 Finally, we remark that our proof does not explicitly use the cup product condition $C_n$. This property is invoked implicitly when we use that the local Galois groups $\Gamma_\q$ satisfy the strong $n$-fold Massey property for $n \geq 3$.
 
\begin{proof} 
For every $\p \in S$, denote by $\theta_\p : \Gamma_\p \rightarrow (\Z/p)^n$ the restriction of $\theta$ to $\Gamma_\p$.

For $\q$ ramified in $\theta$, recall that $\Gamma_\p$ is either a Demushkin group or free pro-$p$. By \cite[Proposition 4.1]{MT_JNT} and \cite[Theorem 4.3]{MT}  Demushkin groups satisfy the strong $n$-fold Massey property
for $n \geq 3$.  
The lifts of $\theta_\q$ to homomorphisms $\rho_\q:\Gamma_\q \to U_{n+1}$
whose existence is guaranteed by \cite{MT_JNT} and \cite{MT}
necessarily have image in $\varphi^{-1}(\theta(\Gamma)) \subset U_{n+1}$.  
These are the Massey local plans referred to at the end of \S\ref{subsection_local}. 

$$\xymatrix{ 
& & & \varphi^{-1}(\theta(\Gamma)) \ar[r]\ar@{^{(}->}[r]\ar[d]\ar@{->>}[d]_\varphi 
& U_{n+1} \ar[d]\ar@{->>}[d]_\varphi\\
\Gamma_\q \ar@{->}@/^-1.50pc/[rrr]_{\theta_\q}\ar[rrru]^{\rho_\q}\ar[r] 
& \Gamma \ar[rru]_{?\rho} \ar[rr]\ar@{->>}[rr]_\theta & & \theta(\Gamma) \ar[r]\ar@{^{(}->}[r] &  (\Z/p)^n    }
$$
We simply apply Theorem~\ref{maintheorem_bis} with $U=\varphi^{-1}(\theta(\Gamma))$ and 
$U/V=\theta(\Gamma)$ to get the existence of $\rho$ which we then compose with the injection $\varphi^{-1}(\theta(\Gamma))
\hookrightarrow U_{n+1}$ to establish the result.
\end{proof}

 \begin{rema}   The method shows that each embedding problem can be solved by a tame prime~$\q$ that is given via the Chebotarev density theorem. One needs at most $n(n-1)/2$ such primes. Using GRH effective versions of Chebotarev's theorem, one can bound the absolute norms.  See Remark \ref{rema_1.7}.
 \end{rema}


\subsection{Abelian plans }

Let $\theta=(\chi_1, \cdots, \chi_n): \gta \rightarrow (\Z/p)^n$ be a homomorphism  in $C_n$:
  $$  \chi_1\cup \chi_2=\cdots = \chi_{n-1}\cup \chi_n=0.$$
The proof of Theorem~\ref{Theorem_Massey} above is {\it not} explicit for the ramified primes of $\theta$. The condition in $C_n$ is also used only in the local results we cite from \cite{MT_JNT} and \cite{MT}.
In this section we give another proof that is 
explicit for these primes when $p>n$ and highlights condition $C_n$.

Let~$S$ be the set of ramification of $\theta$, which  is by the choice of $\Gamma^{ta}_K$ tame. 
Then for  $\p \in S$ we have $\zeta_p\in K_\p$. 
 For $\psi \in H^1(\gta)$, set $\psi_\p:=\psi |_{\Gamma_\q}$.

 \begin{lemm} \label{lemma_notclose}  Suppose $\chi_{i,\p} \neq 0$.
 Then there exists $\lambda_{\q,i} \in \Z/p$ such that $\chi_{i+1,\q}=\lambda_{\q,i} \chi_{i,\q}$. 
 \end{lemm}

 \begin{proof} 
 The arguments below are standard in local Galois cohomology. Using the local Euler-Poincar\'{e} characterstic and 
 that $\zeta_p \in K_\q$ one has
 $$\dim H^i(\Gamma_\q,\Z/p) = \dim H^i(\Gamma_\q,\mu_p)
 =\left\{ \begin{array}{ll}
1 & i=0 \\
2 & i=1\\
1 &i=2 \end{array}\right. .$$
As $\mu_p \simeq \Z/p$ in $K_\q$, the  perfect local pairing becomes
 $H^1(\Gamma_\q,\Z/p) \times H^1(\Gamma_\q,\Z/p) \to H^2(\Gamma_\q,\Z/p)$.
 That
  $\dim H^1(\Gamma_\q)=2$   gives that $\chi_{i,\q}$ is its own annihilator under the local pairing. The result follows from the condition
   $\chi_{i,\q} \cup \chi_{i+1,\q} =0$.
 \end{proof}

We need to lift 
$\theta:\gta \to (\Z/p)^n 
\twoheadleftarrow U_{n+1}$ to a 
homomorphism $\gta \to U_{n+1}$. We will do this in separate blocks of (local) nonzero characters. If $\chi_{j,\q}=\chi_{j+1,\q}=\cdots= \chi_{j+k,\q}=0$ we simply take $\rho_\q$ to be the trivial lift to $U_{n+1}$ on this block.

For a block with nonzero characters,
$\chi_{j,\q}, \chi_{j+1,\q}, \cdots, \chi_{j+k,\q}$, we have
\begin{eqnarray*}
\chi_{j+1,\q} & = &\lambda_{j,\q} \chi_{j,\q},\\
\chi_{j+2,\q} & = &\lambda_{j+1,\q} \chi_{j+1,\q},\\
 &\cdots \\
\chi_{j+k,\q} & = &\lambda_{j+k-1,\q} \chi_{j+k-1,\q}.
\end{eqnarray*}
On this block we set
$\rho_\q(\sigma_\q) $ and $\rho_\q(\tau_\q)$ to be  elements of $U_{n+1}$  that are nonzero only on the diagonal and on the `near-diagonal', that is at $(i,i)$ and $(i,i+1)$ entries.
E.g., starting with $\theta =(\chi_1,\chi_2,\chi_3)$ and
 $\chi_{2,\q} =\lambda_{1,\q} \chi_{1,\q}$ and $\chi_{3,\q} =\lambda_{2,\q} \chi_{2,\q}$
our local plan is, for $\gamma \in \{\sigma_\q,\tau_\q\}$:
$$\rho_\q(\gamma):=
\left(\begin{array}{cccc}
1 & \chi_{1,\q}(\gamma) & 0 & 0\\
0 & 1& \lambda_{1,\q} \chi_{1,\q}(\gamma) & 0 \\
0 & 0& 1 &\lambda_{1,\q}\lambda_{2,\q}\chi_{1,\q}(\gamma)\\
0 & 0 & 0 &1
\end{array}\right).
$$
We will show $\rho_\q(\sigma_\q)$ and $\rho_\q(\tau_\q)$  commute. From the definition of $\Lambda_\q$ below, we have 
 $\Lambda_\q^n=0$. Since $p>n$, $\rho_q$ gives
a representation of $\Gamma^{ab}_\q$, the abelianization of
$\Gamma_\q$ which is our local plan. 

We now show the commutation result. Set 
$$ 
\Lambda_\q=\left(\begin{array}{cccc}
0 & 1 & 0 & 0\\
0 & 0& \lambda_{1,\q}  & 0 \\
0 & 0& 0 &  \lambda_{1,\q}\lambda_{2,\q}\\
0 & 0 & 0 &0
\end{array}\right)$$
so 
\begin{eqnarray*} \rho_\q(\sigma_\q) \rho_\q(\tau_q) & =& \left(I+\chi_{1,\q}(\sigma_q)\Lambda_\q\right)
(I+\chi_{1,\q}(\tau_q)\Lambda_\q) \\
&  = &
I+(\chi_{1,\q}(\sigma_\q) +\chi_{1,\q}(\tau_\q)) \Lambda_\q + \chi_{1,\q}(\sigma_q)\chi_{1,\q}(\tau_\q) \Lambda^2_\q\\
&  = &
I+(\chi_{1,\q}(\tau_\q) +\chi_{1,\q}(\sigma_\q)) \Lambda_\q + \chi_{1,\q}(\tau_q)\chi_{1,\q}(\sigma_\q) 
\Lambda^2_\q \\
& =& \rho_\q(\tau_\q) \rho_\q(\sigma_\q).
\end{eqnarray*}

We have proved:
 
\begin{coro} Let $n\geq 3$, $p>n$ and $\zeta_p \notin K$. 
Then the strong Massey property  holds for~$\theta$ and $\gta$ with $\rho_\q$ as above constructed in blocks. The element
$\rho_\q(\tau_\q)$ has order $p$ in the lift $\rho:\gta \to U_{n+1}$ of $\theta$.
\end{coro}


\subsection{More liftings}

Set $r\geq 1$. 
Let $Gl_{n+1}(\Z/p^r)$ be the group of invertible  $(n + 1)\times (n + 1)$-matrices with entries in $\Z/p^r$ and
$U_{n+1}(\Z/p^r) \subset Gl_{n+1}(\Z/p^r)$ be the the subgroup of all upper-triangular unipotent matrices. Let 
$\pi_r : Gl_{n+1}(\Z/p^r)\rightarrow Gl_{n+1}(\Z/p) $  be the mod $p$ reduction homomorphism. It is well known that $Ker(\pi_r)$ is a $p$-group.
Let $U \subset Gl_{n+1}(\Z/p^r)$ be a $p$-group. The Scholz-Reichardt Theorem gives the existence of a Galois extension $K$ over $\Q$ such that $Gal(K/\Q) \simeq U$. In this case, if $p^m$ is the exponent of $U$,
we can guarantee  every ramified prime $\p$ satisfies   $N(\p) \equiv 1$ modulo $p^m$ and so all ramification is tame. 

 \smallskip

 On the other hand, by following the Massey product philosophy,  starting with  $\theta : \Gamma \rightarrow (\Z/p)^n$  in $C_n$, one can ask if $\theta$ lifts to 
 a  $\rho_r:\Gamma \to Gl_{n+1}(\Z/p^r)$ such that the diagram below commutes:

$$\xymatrix{ &  \pi_r^{-1} \left(U_{n+1}(\F_p)\right) \subset Gl_{n+1}(\Z/p^r) \ar@{->}[d]^{\pi_r} \\
&U_{n+1}(\F_p) \ar@{->}[d]^{\varphi}  \\
\Gamma    \ar@{->}[r]_{\theta} \ar@{->}[ru]_{\rho} \ar@{->}[ruu]^{? \rho_r} & (\Z/p)^n
 }
$$
Here  $\rho$ is a lift  given by Theorem \ref{Theorem_Massey}.  Since $Ker(\pi_r)$ is a $p$-group, we have $\rho_r(\Gamma)$ is also a $p$-group.
 
 \begin{theo} Take $\Gamma=\gta$ or $\ga$, and suppose $\zeta_p\notin K$. 
 For $n\geq 3$, let $\theta : \Gamma \rightarrow (\Z/p)^n$  satisfy  $C_n$.
 Let
 $\rho$ be  given by Theorem \ref{Theorem_Massey}, where we choose all 
 tame primes $\q'$ from that proof to satisfy $N(\p') \equiv 1$ modulo $p^{m(r)}$, where $p^{m(r)}$ is the exponent of $U_{n+1}(\Z/p^r)$. This is possible as $\zeta_p \notin K$.
  \\
 (i) Then for every $r\geq 1$, there exists a homomorphism $\rho_r : \Gamma \rightarrow Gl_{n+1}(\Z/p^r)$  such that 
 $\pi_r \circ \rho_r = \rho$
 and $ \theta=\varphi\circ  \pi_r\circ  \rho_r$.\\
 (ii) If moreover $\zeta_{p^{r}}\in K_\p$ for every ramified prime $\p$ in $\theta$
 then  $\rho_r$ can be taken such that $\rho_r(\Gamma) \subset U_{n+1}(\Z/p^r)$.
 \end{theo}

 \begin{proof} (i) Let $S$ be the set of ramified primes of $\theta$. 
By \cite[Proposition 4.1]{MT_JNT} and \cite[Theorem 4.3]{MT}, we may choose for each prime  $\p \in S$  
a lift $\rho_\p:\Gamma_\q \to U_{n+1}(\F_p)$. Using Theorem~\ref{Theorem_Massey} 
we realize a global lift $\rho: \Gamma \rightarrow U_{n+1}(\F_p)$ of $\theta$ 
whose restrictions to $\Gamma_\p$
 for all $\p \in S$ are $\rho_\p$.

\smallskip

We have to add many new ramified primes $\q'$ to obtain $\rho$. 
As $\zeta_p \notin K$, they can be chosen such that $N(\p') \equiv 1$ modulo $p^{m(r)}$, where $p^{m(r)}$ is the exponent of $U_{n+1}(\Z/p^r)$. 
We give each such  prime $\q'$ the [S-R] local plan, that is
\begin{itemize}
    \item[$-$] $\rho_{r,\q'}(\sigma_{\q'})=1$, and
    \item[$-$] $\rho_{r,\q'}(\tau_{\q'})$ is any lift  of  $\rho_{\q'}(\tau_{\q'})=\overline{x} \in U_{n+1}$, to $U_{n+1}(\Z/p^r)$. This element is killed by $p^{m(r)}$ and by local class field theory the image of $\tau$ in $\Gamma^{ab}_\q$ has order at least $p^{m(r)}$.
    
\end{itemize}

\smallskip

It remains to show the existence,
for $\q\in S$, of local plans $\rho_{r,\p} : \Gamma_\p \rightarrow  Gl_{n+1}(\Z/p^r)$ whose reductions modulo $p$ are
$\rho_\p$.

First,  there is the trivial local plan: when $\p|p$ and $\zeta_p\notin K_\p$. As $\Gamma_\p$ is free pro-$p$,  any lift of $\rho_\p(\Gamma_\p)$ in $U_{n+1}(\Z/p^{r})$ works.

For the other primes $\p \in S$  one needs more local lifting results. One uses the local plans given by:
\begin{itemize}
 \item[$-$] B\"ockle \cite[Theorem 1.3]{Bo} for the tame primes ($\p \nmid p$),
\item[$-$]  Emerton and Gee \cite[Theorem 6.4.4]{EG} for the wild primes ($\p |p$).
\end{itemize}

In \cite{Bo} and   \cite{EG}, the authors prove the existence of lifts $\rho_{\infty,\p}$ into $Gl_{n+1}(\Z_p)$ for every representation $\Gamma_\p \rightarrow Gl_{n+1}(\F_p)$. Applying these results to   $\rho_\q:\Gamma_\q \to U_{n+1}(\F_p)$, $\p \in S$
and reducing  modulo $p^r$ 
gives the desired local plans. Now (i)
follows by Theorem~\ref{maintheorem_bis}
with $U:=\pi_{r}^{-1}\left(\rho(\Gamma) \right)$ and $V=Ker(\pi_r) \cap U$.

For (ii) assume that $\zeta_{p^{r}}\in K_\p$ and since $\rho_\p(\Gamma_\p) \subset U_{n+1}(\F_p)$, a recent result of Conti, Demarche and Florence \cite{CDF} shows that there exist local lifts $\rho_{r,\p}$ of $\rho_\p$ in $Gl_{n+1}(\Z/p^r)$, that can be taken with image in 
$U_{n+1}(\Z/p^r)$, in the tame and wild setting.
Set $U:=\pi_{r}^{-1}\left(\rho(\Gamma) \right) \cap U_{n+1}(\Z/p^r)$ and $V=Ker(\pi_r) \cap U$. Since $\rho(\Gamma) \simeq U/ V$, and $\rho(\Gamma)$ and  $V$ are $p$-groups, we see $U$ is also a $p$-group. 
We apply 
Theorem \ref{maintheorem_bis} with $U$, $U/V$, and the above local plans $\rho_{r,\p}$.
 \end{proof}

 \begin{rema}Our construction does not allow us to pass to the projective limit to get a lift in $Gl_{n+1}(\Z_p)$. This is because $m(\infty)=\infty$ and we would need to choose primes $\q'$ with $N(\q') \equiv 1$ modulo $p^\infty$.
 
 \end{rema}

 \begin{rema} Observe that in the nondegenerate case, the group $\rho_r(\Gamma)$ contains $U_{n+1}(\Z/p^r)$. \end{rema}
 
\smallskip

 To conclude let us show why the condition "$\zeta_{p^{r}} \in K_\p$" given in \cite{CDF} is in a certain sense necessary.

 \begin{prop} Let $\p$ be a tame prime. Suppose given a   homomorphism $\rho_\p : \Gamma_\p \rightarrow U_{n+1}(\F_p)$, and a lift $\rho_{\p,r}$ of $\rho_\p$ in $U_{n+1}(\Z/p^r)$:
$$\xymatrix{ 
&U_{n+1}(\Z/p^r) \ar@{->}[d]^{\pi_r}  \\
\Gamma_\p    \ar@{->}[r]_{\rho_\p} \ar@{->}[ru]^{\rho_{\p,r}} & U_{n+1} \\ 
 }
$$
If  a character $\theta_i$ on the near diagonal of $U_{n+1}$ is ramified, then $\zeta_{p^{r}} \in K_\p$. That is, if $\theta_i(\tau_\p) \neq 0$,  then $N(\p) \equiv 1$ modulo~$p^r$.
 \end{prop}

 \begin{proof} By hypothesis there exists $\theta_i \in H^1(\Gamma_\p,\Z/p)$ such that $\theta_i(\tau_\p) \neq 0$. But this homomorphism  lifts to $\theta_{i,r} \in H^1(\Gamma_\p,\Z/p^r)$.  The two corresponding extensions are {\it cyclic extensions}, included in each other, and since $\theta_i$ is ramified (at $\p$), it forces the cyclic degree $p^r$ extension associated to $\theta_{i,r}$ to be totally ramified,  which implies $\zeta_{p^r} \in K_\p$ by class field theory.
 \end{proof}


\subsection{Finite ramification sets}
Let $S$ be a finite set of primes of $K$ and set $K_S$  to be the maximal pro-$p$ extension of~$K$ unramified outside $S$. When $p=2$, we assume that the real archimedean places remain real in every subfield of $K_S$. Set $\Gamma_S:=Gal(K_S/K)$.  Shafarevich and Koch showed  the pro-$p$ group $\Gamma_S$ is finitely generated. 

\subsubsection{Free and Demushkin groups} Let $S_p$ be the set of $p$-adic primes of $K$. The pro-$p$ group $\Gamma_{S_p}$ can be free.
But, as observed first by Shafarevich,   $\Gamma_{S_p}$ can be a free noncommutative pro-$p$ group, for instance  when $p$ is regular, and $K=\Q(\zeta_p)$: in this case $\Gamma_{S_p}$ is free on $(p+1)/2$ generators. When  $\Gamma_{S_p}$ is free, it  obviously satisfies the strong $n$-fold Massey property for every $n\geq 2$.
We state a Conjecture of Gras \cite[Conjecture 7.11]{Gras-CJM}: 
\begin{conjectureNC}[Gras]
Fix a number field  $K$. For $p\gg 0$ the pro-$p$ group $\Gamma_{S_p}$ is free on $r_2+1$ generators, where $2r_2$ is the number of complex embeddings of $K$.
\end{conjectureNC}

There is another context for which $\Gamma_{S_p}$ satisfies the strong $n$-fold Massey property for every $n\geq 3$: when $\Gamma_{S_p}$ is Demushkin. This situation has been studied in \cite{Win}, Section~3.

\subsubsection{Deep relations}

When $S\cap S_p=\emptyset$, the pro-$p$ group $\Gamma_S$ is FAB: every open subgroup has finite abelianization. 

Observe first that $\Gamma_S$ can be trivial. Indeed, take $K=\Q$, and $S=\emptyset$. 
It can also be cyclic of order $p^m$.  
Indeed, take $K=\Q$, $p$ odd, and $\ell$ a prime such that $v_p(\ell-1)=m$. Set $S=\{\ell\}$; then $\Gamma_S \simeq \Z/p^m$. 
In this situation,  it is not difficult to see that $\Gamma_S$ satisfies the strong $n$-fold Massey property if and only if $n +1 \leq p^m$. \
In the cyclic setting, $\Gamma_S$ is presented by one generator $x$ and one relation $r:=x^{p^m}$  of depth $p^m$ (using the Zassenhaus filtration).

There is the following general result of Vogel \cite[Corollary 1.2.9]{Vogel_phd}:

\begin{theo} \label{theo_vogel}  Let $G$ be a finitely generated pro-$p$ group described by generators and a set $R$  of relations. If all elements of $R$ are of at least depth  $n+1$, then $G$ satisfies the strong $k$-fold Massey property for $2\leq k \leq n$.\end{theo}

\begin{rema} 
We use the terminology as in \cite{Vogel}.
    Let $G$ be a pro-$p$ group and $n\geq 2$ an integer. Suppose that $G=F/R$ where $F$ is a free pro-$p$ group on generators $x_1,\ldots,x_n$, and $R\subseteq F^p[F,F]$. The following  are equivalent:
    \begin{itemize}
        \item[i)] $R\subseteq F_{(n+1)}$.
        \item[ii)] All $k$-fold Massey products are stricly and uniquely defined and equal to 0, for $2\leq k\leq n$.
        \item[iii)] $G$ satisfies the strong $k$-fold Massey vanishing property for $2\leq k\leq n$.
        \item[iv)] $G$ satisfies the  $k$-fold Massey vanishing property for $2\leq k\leq n$.
    \end{itemize}
\end{rema}
\begin{proof}
The implication from $i)$ to $ii)$ follows from \cite[Theorem A3]{Vogel}.

The implications from $ii)$  to $iii)$ and from $iii)$ to $iv)$ are clear.

Now we suppose that $iv)$ holds. Let $\chi_1,\ldots,\chi_n\in H^1(F,\F_p)=H^1(G,\F_p)$ be the dual basis to  $x_1,\ldots,x_n$.
   That $G$ satisfies the  $2$-fold Massey vanishing property means that all cup products $\chi_{i_1}\cup \chi_{i_2}$ are zero, for $1\leq i_1,i_2\leq n$. By \cite[Theorem A3]{Vogel},
 for every $f\in R$ and every $1\leq i_1,i_2\leq n$, $I=(i_1,i_2)$, one has
   \[\epsilon_{I,p}(f)=(-1)^{2-1}{\rm tr}_f\langle \chi_{i_1},\chi_{i_2}\rangle =0.\]
   This implies that $f\in F_{(3)}$ by \cite[Lemma 2.19]{Vogel}, and hence $R\subseteq F_{(3)}$. 
   
   Now because $R\subseteq F_{(3)}$, we see that   for all $1\leq i_1,i_2,i_3\leq n$, triple Massey products $\langle \chi_{i_1},\chi_{i_2},\chi_{i_3}\rangle$ are well defined, by \cite[Theorem A3]{Vogel}. Thus $\langle \chi_{i_1},\chi_{i_2},\chi_{i_3}\rangle=0$ because $G$ satisfies the  $3$-fold Massey vanishing property. 
   Also by \cite[Theorem A3]{Vogel}, for every $f\in R$ and every $1\leq i_1,i_2,i_3\leq n$, $I=(i_1,i_2,i_3)$, one has
   \[\epsilon_{I,p}(f)=(-1)^{3-1}{\rm tr}_f\langle \chi_{i_1},\chi_{i_2},\chi_{i_3}\rangle =0.\] 
   This implies that $f\in R_{(4)}$ by \cite[Lemma 2.19]{Vogel}. Hence $R\subseteq R_{(4)}$.
   Continuing in this way, we show that $R\subseteq F_{(k+1)}$ for all $2\leq k\leq n$. In particular, $R\subseteq F_{(n+1)}$.
\end{proof}

To conclude, we give another situation where we can apply Theorem \ref{theo_vogel} .
 Take $T$ a finite set of primes of $K$, disjoint from $S$. Let $K_S^T$ be the maximal pro-$p$ extension of $K$ unramified outside $S$, with the primes of $T$ splitting completely in $K_S^T$. Set $\Gamma_S^T:=Gal(K_S^T/K)$.

 \begin{coro} Take $n\geq 3$. Let $K$ be a number field, not totally real, satisfying  Gras's conjecture.  Then for $p\gg 0$, there exists a set $T$ of primes of $K$, coprime to $p$, such that the pro-$p$ group $\Gamma_{S_p}^T$ is  infinite, \
 has finite abelianization 
 and satisfies the strong $n$-fold Massey property. 
\end{coro} 

\begin{proof} We apply the strategy of \cite{HMR}:  We may take the quotient of the free pro-$p$ group  $\Gamma_{S_p}$ (Gras' conjecture) by  any Frobenius elements whose depth in the Zassenhaus filtration is greater than $n+1$,  by $p^{n}$-powers of Frobenius elements that generate  $\Gamma_{S_p}$,  and apply Theorem~\ref{theo_vogel}. Chebotarev's theorem gives a positive density of such primes for our set~$T$.
\end{proof}


\end{document}